\documentclass[reqno,12pt]{amsart}
\usepackage{times,cite}

\usepackage[]{graphicx}

\usepackage{mathrsfs}

\newtheorem{thm}{Theorem}[section]

\newtheorem{cor}[thm]{Corollary}

\newtheorem{prop}[thm]{Proposition}

\newtheorem{rmk}[thm]{Remark}

\usepackage{cite}
\usepackage[cp1250]{inputenc} 

\usepackage{a4wide} 
\usepackage[stable]{footmisc} 
\usepackage{fancyhdr} 
\usepackage{layout} 

\usepackage{amsthm,amsmath,amstext,amsbsy,amsfonts,amssymb} 
\usepackage{stmaryrd} 
\usepackage{turnstile} 
\usepackage{mathtools} 

\usepackage{graphicx} 
\usepackage[all,cmtip]{xy} 
\usepackage{fancybox} 
\usepackage{hhline} 
\usepackage{colortbl} 
\usepackage{color} 
\usepackage{pstricks} 
\usepackage{qtree} 
\usepackage{tikz} 

\usepackage{enumerate} 
\usepackage[ps2pdf]{hyperref} 
\usepackage{index} 
\usepackage{longtable} 
\usepackage{cite} 

\usepackage{ifthen} 
\usepackage{xstring} 
\usepackage{calc} 
\usepackage{forloop} 

\usepackage{xspace} 
\ifx\uplodadedmylogic\undefined
	\def\uplodadedmylogic{}

\newcommand{\str}[1]{\StrLen{#1}[\tmplength]\ifthenelse{1=\tmplength{}}{\mathcal{#1}}{\underline{#1}}}

\newcommand{\A}{\str{A}}

\newcommand{\M}{\str{M}}

\newcommand{\stru}[2]{\lara{#1,#2}} 
\newcommand{\lang}[1]{\lara{#1}} 

\newcommand{\thn}[1]{\mathrm{#1}} 

\newcommand{\PA}{\thn{PA}}
\newcommand{\PrA}{\thn{Pr}}

\newcommand{\LA}[1][{}]{\thn{LA}_{#1}}

\newcommand{\DeLOs}[2]{{\tiny\ifthenelse{\equal{#1}{#2}}{\thn{DeLO}^{#1}}{\ifthenelse{\equal{#2}{}}{\thn{DeLO}^{#1}}{\thn{DeL0}^{\vpair{$#1$}{$#2$}}}}}}

\newcommand{\lekv}{\leftrightarrow}
\newcommand{\limp}{\rightarrow}
\renewcommand{\land}{\,\&\,}

\newcommand{\IFF}{\Leftrightarrow}
\newcommand{\PAK}{\Rightarrow}

\newcommand{\Th}[1]{Th(#1)}

\newcommand{\landef}[2][{}]{\csname Ln#2\endcsname{#1} = \csname L#2\endcsname{#1}}
\newcommand{\landefl}[2][{}]{$\landef[#1]{#2}$, kde \csname Ld#2\endcsname{#1}}

\def\LnfGrpA/{aditivní jazyk teorie grup}
\def\LnFGrpA/{Aditivní jazyk teorie grup}

\def\LnfOrd/{jazyk teorie uspořádání}
\def\LnFOrd/{Jazyk teorie uspořádání}

\def\LnfPrv/{výrokový jazyk s množinou prvovýroků $\Prv$}
\def\LnFPrv/{Výrokový jazyk s množinou prvovýroků $\Prv$}

\def\LnfAr/{jazyk aritmetiky}
\def\LnFAr/{Jazyk aritmetiky}

\def\LnfArZ/{jazyk $\Z$-aritmetiky}
\def\LnFArZ/{Jazyk $\Z$-aritmetiky}

\def\LnfArPA/{jazyk Peanovy aritmetiky}
\def\LnFArPA/{Jazyk Peanovy aritmetiky}

\def\LnfArZPA/{jazyk $\Z$-Peanovy aritmetiky}
\def\LnFArZPA/{Jazyk $\Z$-Peanovy aritmetiky}

\def\LnfArA/{aditivní jazyk aritmetiky}
\def\LnFArA/{Aditivní jazyk aritmetiky}

\def\LnfArAA/{jazyk aditivní aritmetiky}
\def\LnFArAA/{Jazyk aditivní aritmetiky}

\def\LnfArZAA/{jazyk $\Z$-aditivní aritmetiky}
\def\LnFArZAA/{Jazyk $\Z$-aditivní aritmetiky}

\def\LnfArLA/{jazyk lineární aritmetiky}
\def\LnFArLA/{Jazyk lineární aritmetiky}

\def\LnfArZLA/{jazyk $\Z$-lineární aritmetiky}
\def\LnFArZLA/{Jazyk $\Z$-lineární aritmetiky}

\def\LnfArKLA/{jazyk $\kappa$-lineární aritmetiky}
\def\LnFArKLA/{Jazyk $\kappa$-lineární aritmetiky}

\def\LnfArKZLA/{jazyk $\kappa$-$\Z$-lineární aritmetiky}
\def\LnFArKZLA/{Jazyk $\kappa$-$\Z$-lineární aritmetiky}

\def\LnfMod/{jazyk modulů nad okruhem}
\def\LnFMod/{Jazyk modulů nad okruhem}

\def\LnfSuc/{jazyk následníka}
\def\LnFSuc/{Jazyk následníka}

\def\LnfSucO/{jazyk následníka s nulou}
\def\LnFSucO/{Jazyk následníka s nulou}

\newcommand{\Prv}{\mathbb{P}}
\def\DNF/{DNF}
\def\CNF/{CNF}

\fi

\ifx\uplodadedmygeneral\undefined
	\def\uplodadedmygeneral{}
\newcommand{\N}{\mathbb{N}}
\newcommand{\Z}{\mathbb{Z}}

\newcommand{\R}{\mathbb{R}}

\newcommand{\lr}[1]{\{#1\}}

\newcommand{\lara}[1]{\langle #1\rangle}

\newcommand{\set}[2]{\lr{#1;#2}}

\newcommand{\vect}[1]{\overline{#1}}
\newcommand{\restr}{\upharpoonright}

\newcommand{\vpairg}[3]{{\newbox\horni \newbox\dolni \setbox\horni=\hbox{#1} \setbox\dolni = \hbox{#2} \newdimen\zdeposun \newdimen\wddolni \newdimen\wdhorni \setlength{\wddolni}{\wd\dolni} \setlength{\wdhorni}{\wd\horni} \setlength{\zdeposun}{-\wdhorni-\wddolni/2+\wdhorni/2} {\lower -#3 \hbox{#1}}\kern\zdeposun {\lower #3 \hbox{#2}}}} 
\newcommand{\vpair}[2]{\vpairg{#1}{#2}{.85ex}}

\fi
\ifx\uplodadedmytextstructuring\undefined
	\def\uplodadedmytextstructuring{}
\newcommand{\idf}[3]{\label{#1}\csname index#2\endcsname[|textbf]#3} 

\newcommand{\refer}[1]{(\ref{#1})}
\newenvironment{piecewise}{\left\{\begin{array}{rl}}{\end{array}\right.}

\newcommand{\benum}{\begin{enumerate}}
\newcommand{\eenum}{\end{enumerate}}

\newcommand{\bitem}{\begin{itemize}}
\newcommand{\eitem}{\end{itemize}}

\newcommand{\uvz}[1]{``#1''} 

\newcommand{\axiominfo}[7]{
\begin{tabular}{#7}
Znění: & #1\\
Význam: & #2
\ifthenelse{\equal{#4}{}}{}{\\#3 & #4}%
\ifthenelse{\equal{#6}{}}{}{\\#5 & #6}%
\end{tabular}
}

\newcommand{\repeatstatement}[4][thm]
{
{
\renewcommand{\thestatement}{\ref{#2}}
\begin{#1}[#3]
#4
\end{#1}
\addtocounter{statement}{-1}
}
}

\newcounter{thislevel}
\newcommand{\mysect}[2]{\setcounter{thislevel}{\value{mysectlevel}} \addtocounter{thislevel}{#1}
\if\arabic{thislevel}0\part{#2}\fi
\if\arabic{thislevel}1\chapter{#2}\fi
\if\arabic{thislevel}2\section{#2}\fi
\if\arabic{thislevel}3\subsection{#2}\fi
\if\arabic{thislevel}4\subsubsection{#2}\fi
\if\arabic{thislevel}5\paragraph{#2}\fi
\if\arabic{thislevel}6\subparagraph{#2}\fi}
\fi

\newcommand{\linlang}[1][]{L^{lin}_{#1}}
\newcommand{\sLA}[1][]{\LA[#1]^{\#}}

\begin{document}

\title{A Wild Model of Linear Arithmetic and Discretely Ordered Modules}

\author[P. Glivick\' y]{Petr~Glivick\' y%
}

\address{Petr Glivick\' y: Academy of Sciences of the Czech Republic, Institute of Mathematics\\ 
\v Zitn\' a 25, 115 67 Praha~1, Czech Republic}
\address{Petr Glivick\' y: University of Economics, Department of Mathematics\\ 
Ekonomick\' a 957, 148 00 Praha~4, Czech Republic}
\author[P. Pudl\' ak]{Pavel Pudl\' ak}

\address{Pavel Pudl\' ak: Academy of Sciences of the Czech Republic, Institute of Mathematics\\ 
\v Zitn\' a 25, 115 67 Praha~1, Czech Republic}

\thanks{The research leading to these results has received funding from the European Research Council under the European Union's Seventh Framework Programme (FP7/2007-2013) / ERC grant agreement n${}^{\circ}$\ 339691. This paper was processed with contribution of long term institutional support of research activities by Faculty of Informatics and Statistics, University of Economics, Prague.}


%
%

\subjclass[2010]{Primary 03C62, 03C45; Secondary 06F25}

\date{February 8, 2016}

\keywords{linear arithmetic, ordered modules, NIP}

\begin{abstract}
Linear arithmetics are extensions of Presburger arithmetic ($\PrA$) by one or more unary functions, each intended as multiplication by a fixed element (scalar), and containing the full induction schemes for their respective languages. 

In this paper we construct a model $\M$ of the $2$-linear arithmetic
$\LA[2]$ (linear arithmetic with two scalars) in which an infinitely long initial
segment of \uvz{Peano multiplication} on $\M$ is
$\emptyset$-definable. This shows, in particular, that $\LA[2]$ is not
model complete in contrast to theories $\LA[1]$ and $\LA[0]=\PrA$ that are known to satisfy quantifier elimination up to disjunctions of primitive positive formulas.

As an application, we show that $\M$, as a discretely ordered module over the discretely ordered ring generated by the two scalars, is not NIP, answering negatively a question of Chernikov and Hils.
\end{abstract}

\maketitle

\section{Introduction}

There is longstanding interest in definability and related properties of various extensions of Presburger arithmetic ($\PrA = \Th{\stru{\N}{0,1,+,\leq}}$) by fragments of multiplication (see \cite{Bes02} for a good survey).
One class of such extensions are linear arithmetics, introduced in \cite{Gli13dis} (but various similar situations were studied much earlier, see \cite{Gli15LAs} for details). For any cardinal number $\kappa$, the $\kappa$-linear arithmetic $\LA[\kappa]$ is an arithmetical theory containing the full induction scheme for its language $\lang{0,1,+,\leq,a_\alpha}_{\alpha\in\kappa}$, where each $a_\alpha$ is a unary function symbol intended (and axiomatized) as multiplication by one fixed element (for the precise definition, see Section \ref{sect:lineararith}).



\medskip

The theory $\LA[0]$ is just $\PrA$. Its definability properties are well understood. In particular, every formula is in $\PrA$ equivalent to a disjunction of bounded primitive positive formulas (bounded pp-formulas; i.e. formulas of the form $(\exists \vect{x}<\vect{t}) \chi(\vect{x},\vect{y})$, where $\chi$ is a conjunction of atomic formulas), hence $\LA[0]$ is model complete. Also it is a decidable theory.
The same properties -- bounded pp-elimination (\cite{Gli13dis} or \cite{Gli15LAs}), (consequently)
model completeness and decidability (\cite{Pen71} and independently \cite{Gli13dis} or \cite{Gli15LAs}) -- have been be shown also for $\LA[1]$.
For $\kappa\geq 2$, nevertheless, $\LA[\kappa]$ was only known to
satisfy quantifier elimination up to bounded formulas
\cite{Gli14boundedQELAs}. 
For more results on model theory of linear
arithmetics, see \cite{Gli15LAs} and \cite{Gli14boundedQELAs}.


\medskip

In this paper, we prove that the theories $\LA[\kappa]$ with
$\kappa\geq 2$ are not model complete. We do this by constructing a
model $\M=\stru{M}{0,1,+,\leq,a,b}\models\LA[2]$ such that for some
$L\in M$ nonstandard, an operation of partial Peano multiplication $\cdot: [0,L]^2\rightarrow M$ is $\emptyset$-definable in $\M$ (see Theorem
\ref{thm:main}). Here, $\emptyset$-definable means definable without
parameters and partial Peano multiplication is an operation $\cdot$ that can be extended to the whole $M^2$ in such a way that $\stru{M}{0,1,+,\cdot,\leq}$ is a model of Peano arithmetic ($\PA$).

Note that, due to the bounded quantifier elimination in $\LA[\kappa]$,
in no model $\M$ of $\LA[\kappa]$ Peano multiplication is definable on the whole $M^2$.

\medskip

As an application of the above result, we show 
in section \ref{sect:nonNIP} that the constructed model $\M\models\LA[2]$ endowed with a natural structure of a (discretely) ordered module 
has the independence property. This answers negatively the question of Chernikov and Hils \cite[Question 5.9.1]{ChH14} whether all ordered modules are NIP.

Let us note that $\LA[1]$ (as well as $\LA[0] = \PrA$) is NIP, which follows easily from the quantifier elimination results in \cite{Gli13dis}, see \cite{Gli15LAs}.

\medskip

Finally, let us remark that an analogous problem of definability of multiplication in expansions of the structure $\stru{\R}{0,+,<}$ has been studied and that it exhibits surprisingly similar features as the problem of definability of multiplication in linear arithmetics (in particular, continued fractions are used in both cases). See \cite{HT14} and \cite{Hie16} for results formally closest to those in this paper.

\section*{Acknowledgments}
The authors want to thank the anonymous referee for a careful reading of the paper and for useful comments that helped improve the quality of this text.

\section{Preliminaries}

\subsection{Linear arithmetics}\label{sect:lineararith}
For any cardinal number $\kappa$, \emph{$\kappa$-linear arithmetic} $\LA[\kappa]$ is the theory in the language $\linlang[\kappa]=\lang{0,1,+,\leq,a_\alpha}_{\alpha\in\kappa}$ (where $a_\alpha$ are unary function symbols) with the following axioms:
%
%
%

\medskip
\def\arraystretch{1.5}
\begin{tabular}{cccc}
(A1) & $0\neq z+1$, & (A2) & $x+1=y+1\limp x=y$,\\
(A3) & $x+0=x$, & (A4) & $x+(y+1)=(x+y)+1$,
\end{tabular}

\begin{tabular}{cc}
(D${}_\leq$) & $x\leq y \lekv (\exists z) (x+z=y)$,
\end{tabular}

\begin{tabular}{cccc}
(L1) & $a_{\alpha}(x+1) = a_{\alpha} x + a_{\alpha} 1$, & (L2) & $a_{\alpha}(a_{\beta}x)=a_{\beta}(a_{\alpha}x)$,
\end{tabular}

\begin{tabular}{cc}
(Ind) & $\varphi(0,\vect{y}) \land (\varphi(x,\vect{y})\limp \varphi(x+1,\vect{y})) \limp (\forall x)\varphi(x,\vect{y})$ for all formulas $\varphi(x,\vect{y})$.
\end{tabular}

\medskip\noindent
\emph{Strictly $\kappa$-linear arithmetic} $\sLA[\kappa]$ is the extension of $\LA[\kappa]$ by the axiomatic scheme

\medskip
\begin{tabular}{cc}
(L${}^{\#}$) & \uvz{$a_{\alpha}$ is not definable by any formula not containing $a_{\alpha}$}.
\end{tabular}

\subsection{Models of $\LA[\kappa]$ as discretely ordered modules}
\label{subsect:modules}
Every $\M\models\LA[\kappa]$ naturally corresponds to a discretely
ordered module over a ring $R$ given by $\set{a_\alpha}{\alpha\in\kappa}$. 
Thus models of linear arithmetics can be understood as certain (in particular satisfying induction) discretely ordered modules. When $\M$ is viewed in this way, we call the elements of the ring $R$ scalars.
Below we describe this correspondence in more detail.

\medskip

$\LA[\kappa]$ proves that all elements are non-negative, but given a model $\M\models\LA[\kappa]$, it is often useful to formally add negative elements to $\M$ and to work with this extension rather than with $\M$ itself. In the rest of this paper we will not explicitly distinguish between these two structures and denote both simply by $\M$. This should not cause any confusion as the correct interpretation will always be clear from the context.

\medskip

In every $\M\models\LA[\kappa]$, multiplication by any polynomial $p\in\Z[a_\alpha]_{\alpha\in\kappa}$ can be naturally defined. Thus $\M$ can be equipped with a structure of an (unordered) $\Z[a_\alpha]_{\alpha\in\kappa}$-module. 
It can be also viewed as an ordered module over the ordered ring 
$$\Z(\M):=\Z[a_\alpha]_{\alpha\in\kappa}/Ann_{\Z[a_\alpha]_{\alpha\in\kappa}}(M),$$ 
where $Ann$ denotes annihilator and the ordering of $\Z(\M)$ is induced by the ordering of $\M$ via the map $[p]\mapsto p1$. 

Let us note that, by induction in $\M$, $Ann_{\Z[a_\alpha]_{\alpha\in\kappa}}(M)=Ann_{\Z[a_\alpha]_{\alpha\in\kappa}}(\{1\})$. Therefore $\Z(\M)=\Z[a_\alpha]_{\alpha\in\kappa}$ (i.e. $\M$ is a faithful $\Z[a_\alpha]_{\alpha\in\kappa}$-module) if and only if all $a_\alpha$'s are algebraically independent over $\Z$ in $\M$.

Notice that if $\M$ has the structure of an $R$-module (for any ring $R$), then the map $r\mapsto r1$ is a homomorphism from $R$ (as a module over itself) to $\M$. We will often identify the scalar $r$ with the element $r1\in M$.

\subsection{Euclidean algorithm and continued fractions}
The proof of our main result is based on certain elementary properties
of continued fractions that are provable in Peano arithmetic. We review all what is necessary here. A more detailed exposition can be found in \cite[Chapters I and II]{Khinchinbook}.

\medskip

Let $\M$ be a model of Peano arithmetic (PA). In this subsection, we will work in $\M$. This means that all quantifications, unless explicitly stated otherwise, are restricted to $M$. In the calculations, however, we will use negative elements and fractions freely.

\medskip

Let us fix $0<b<a\in M$. The Euclidean algorithm starting from $(a,b)$ produces the division chain
\begin{equation}
\label{rrek}
r_{i-2}=r_{i-1}a_i+r_i,
\end{equation}
for $i=0,\ldots,n$, $n\in M$, where $r_{-2}=a$,
$r_{-1}=b>r_0>r_1>\dots>r_n=0$, $r_{n-1}=\gcd(a,b)$, and
$$[a_0;a_1,\ldots,a_n]=a_0+\frac{1}{a_1+\frac{1}{\ddots+\frac{1}{a_n}}}=\frac{a}{b}$$ is the continued fraction of $\frac{a}{b}$.

The numerators and denominators of the convergents $\frac{v_i}{u_i}=[a_0;\ldots,a_i]$ (in the lowest terms) satisfy the recursive relations
\begin{equation}
\label{uvrek}
u_i=u_{i-1}a_i+u_{i-2};\ \ \ \ v_i=v_{i-1}a_i+v_{i-2},
\end{equation}
for $i=0,\ldots,n$, where we set $u_{-1}=v_{-2}=0$ and
$u_{-2}=v_{-1}=1$ (see \cite[Theorem 1]{Khinchinbook} for a proof). Clearly, $0<u_0<u_1<\dots<u_n=b$ and $0\leq
v_0<v_1<\dots<v_n=a$.

From \refer{rrek} and \refer{uvrek} it follows
\begin{equation}
\label{eq:riuivi}
r_i = (-1)^i(a u_i-b v_i),
\end{equation}
for $i=-2,\ldots,n$, 
which can be rewritten as 
$$
\frac{a}{b}-\frac{v_i}{u_i}=(-1)^i\frac{r_i}{b u_i}.
$$
From the last equation we can conclude that
$$
\frac{v_0}{u_0}<\frac{v_2}{u_2}<\cdots< \frac{v_n}{u_n} = \frac{a}{b} < \cdots \frac{v_3}{u_3}<\frac{v_1}{u_1}
$$
and that $|\frac{a}{b}-\frac{v_i}{u_i}|$ is decreasing (see also \cite[Theorem 4]{Khinchinbook}). 

Not only the convergents approximate $\frac{a}{b}$ in this sense, they are exactly all the \uvz{best
  rational approximations of $\frac{a}{b}$ of the second kind},
i.e. the following holds true (provably in PA):
\begin{prop}(see e.g. \cite[Theorems 16 and 17]{Khinchinbook})\\
\label{prop:brapp}
For $u,v\in M$, $u>0$, the following are equivalent:
\benum[1)]
\item $|au-bv|<|au'-bv'|$ for all $\frac{v'}{u'}\neq \frac{v}{u}$ with $0<u'\leq u$.
\item There is $0\leq i\leq n$ such that $\frac{v}{u}=\frac{v_i}{u_i}$.
\eenum
with the exception of $\frac{a}{b}=a_0+\frac{1}{2}$, for which only $1)\PAK 2)$.
\end{prop}

The above proposition gives a definition of the set of all convergents $\frac{v_i}{u_i}$ of $\frac{a}{b}$ by a bounded formula that uses
multiplications only by $a$ and $b$. 
We want a similar definition of the set of pairs $(u_i,v_i)$. Therefore we prove:


\begin{cor}
\label{corapprox}
Assume $\frac{a}{b} \neq a_0+\frac{1}{2}$.
Then for $u,v\in M$, $u>0$, the following are equivalent:
\benum
\item[0*)] $|au-bv|<|au'-bv'|$ for all $(v',u')\neq (v,u)$ with $0<u'\leq u,b$ and $0\leq v' \leq a$.
\item[1*)] $|au-bv|<|au'-bv'|$ for all $(v',u')\neq (v,u)$ with $0<u'\leq u$.
\item[2*)] There is $0\leq i\leq n$ such that $v=v_i$ and $u=u_i$.
\eenum
\end{cor}

\begin{proof}
\uvz{1${}^*$)$\IFF$ 2${}^*$)}: We have the following chain of
equivalences
\[ 
\mbox{1${}^*$ $\IFF$ 1 $\land$ $u,v$ are coprime $\IFF$ 2 $\land$ $u,v$
are coprime $\IFF$ 2${}^*$,} 
\]
where the second equivalence is from Proposition \ref{prop:brapp}, and
the other two are trivial.

\medskip
\uvz{1${}^*$)$\PAK$ 0${}^*$)} is obvious. 

\medskip
\uvz{0${}^*$)$\PAK$ 1${}^*$)}: Let $(v',u')\neq (v,u)$ and $0<u'\leq
u$. We prove $|au-bv|<|au'-bv'|$. First observe that  for
$(v',u')=(a_0,1)$, we get from 0${}^*$
\begin{equation}
\label{eq:prcorapprox}
|au-bv|<|a1-ba_0|=r_0<b.
\end{equation}
We may assume $v'>0$ (otherwise $|au'-bv'|=au'\geq a > b$, and we are done due to \refer{eq:prcorapprox}).
Further, let $v'=ka+v''$, $u'=lb+u''$ with $0<v''\leq a$, $0<u''\leq b$, $k,l\in M$ (notice also that $u''\leq u'\leq u$). Then $|au'-bv'|=|au''-bv''+(l-k)ab|$.
We distinguish the following cases:
\begin{itemize}
\item $l-k=0$: Then $|au''-bv''+(l-k)ab|=|au''-bv''|>|au-bv|$ by 0${}^*$).
\item $l-k\geq 1$: Then $|au''-bv''+(l-k)ab|\geq a > b > |au-bv|$, where the last inequality is due to \refer{eq:prcorapprox} and the first one due to $au''-bv''+(l-k)ab\geq a-ba+(l-k)ab\geq a$.
\end{itemize}
\end{proof}

\section{Wild models of linear arithmetics}

In this section we will construct a model of the arithmetic $\LA[2]$
in which an infinite initial segment of a Peano multiplication is
definable without parameters. In fact, we will prove even a little bit
more. Say that a \emph{formula is $b$-bounded} if all quantifiers in
the formula are of the form $\exists x < b1$, $\forall x < b1$. For
the sake of simplicity, in this subsection we will write $x<b$ instead
of $x< b1$ in the quantifier bounds.

\begin{thm}
\label{thm:main}
Let $\M$ be a non-standard model of $\PA$ and $L\in M$. Let
$\M^+$ be the additive part of $\M$ and ${\cdot}$ the operation
of multiplication in $\M$. 
Then there are elements $a<b\in M$ such that $\cdot\restr [0,L]^2$ is
$\emptyset$-definable by a $b$-bounded formula in $\stru{\M^+}{a,b}$, where $a,b$ stand for unary functions of multiplication by elements $a,b$.
\end{thm}

Note that if $L$ is non-standard, then $\stru{\M^+}{a,b}\models\sLA[2]$ follows automatically for any $a,b$ which satisfy the rest of the statement. Indeed, if one of the scalars were definable from the other (say $b$ from $a$) then the multiplication on $[0,L]^2$ would be definable in $\stru{\M^+}{a}\models\LA[1]$, which contradicts the pp-elimination in $\LA[1]$. (For any pp-formula $\varphi(\vect{x})$ which defines an infinite set, it is easy to find $\vect{u}\neq \vect{v}$ such that $\varphi$ holds for $\vect{u},\vect{v}$ and $\frac{\vect{u}+\vect{v}}{2}$. But for $\varphi(x,y,z)$ defining the graph of multiplication over the diagonal of $[0,L]^2$, this is clearly impossible.)

\medskip

We will prove  Theorem \ref{thm:main} in two steps. 
First, in subsection \ref{sect:modLA3}, we find three elements $a,b,c\in M$ such that $\cdot\restr [0,L]^2$ is $\emptyset$-definable in $\stru{\M^+}{a,b,c}\models\LA[3]$ ($a,b,c$ here again standing for scalar multiplication by elements $a,b,c$). Then, in subsection \ref{sect:modLA2}, we show that we can equivalently replace the scalars $a,b,c$ in the definition of the multiplication by certain functions definable from only two scalars $a',b'$, and thus obtain a $\emptyset$-definition of $\cdot\restr [0,L]^2$ in $\stru{\M^+}{a',b'}\models\LA[2]$.

\subsection{A wild model of $\LA[3]$}
\label{sect:modLA3}

Let a non-standard $\M=\stru{\M^+}{\cdot}\models\PA$ and $L\in M$ be arbitrary. We will find elements $a,b,c\in\M$ such that $\cdot\restr[0,L]^2$ is $\emptyset$-definable in the extension $\stru{\M^+}{a,b,c}\models\LA[3]$ of $\M^+$ by scalar multiplication by $a,b,c$. Note that any choice of $a,b,c$ yields a model of $\LA[3]$, so we only need to take care about definability of the multiplication.

%

\medskip

Let us explain the idea of our construction before going to the details. 
First, we can make things a bit easier by recalling the well-known fact
that it is enough to define in $\stru{\M^+}{a,b,c}$ the function
$x\mapsto x^2$ on the domain [0,2L]. Given the squaring function,
multiplication on $[0,L]^2$ is defined by 
\begin{equation}
\label{eq:multfromsquare}
x\cdot y=\frac{(x+y)^2-x^2-y^2}{2}.
\end{equation}

Let $(z_i)_{i=0}^{4L-1}=(1,1^2,2,2^2,\ldots,2L,(2L)^2)$ represent the required initial segment of $x\mapsto x^2$.
%
%
The idea is then to pick the scalar $c$ as any prime number bigger
than $(2L)^2$ and define $a$ and $b$ in such a way that the numerators
$v_i$, $i<4L$, of convergents of the continued fraction
$[a_0;\ldots,a_{4L-1}]$ representing $a/b$ encode the initial segment of $x\mapsto x^2$ in the following sense:

\begin{equation}
\label{seqzi}
z_i = v_i\ \mathrm{mod}\ c,
\end{equation}
for $i=0,\ldots,4L-1$.

By Corollary \ref{corapprox}, the set $\set{v_i}{i<4L}$ is
$\emptyset$-definable in $\stru{\M^+}{a,b,c}$ (in fact in
$\stru{\M^+}{a,b}$). Then combining this definition with \refer{eq:multfromsquare} and \refer{seqzi} easily gives the sought definition of multiplication on $[0,L]^2$.

\medskip

Now we describe the construction in detail. As mentioned above, we
choose $c$ to be any prime bigger than $(2L)^2$. In order to define
$a$ and $b$, we recursively choose the coefficients $a_i$ of the
continued fraction $[a_0;\ldots,a_{4L-1}]$ in such a way that
\refer{seqzi} holds true for every $0\leq i<4L$ 
with the numerators $v_i$ computed from $a_i$s using
\refer{uvrek}. Then we take $a$ and $b$
coprime such that $a/b=[a_0;\ldots,a_{4L-1}]$.  

Let $0\leq i <4L$ and suppose that we have already defined $a_j$ for all $0\leq j<i$ in such a way that \refer{seqzi} holds. 
Notice that no $v_j$ with $-1\leq j<i$ is divisible by $c$, since $v_{-1}=1$ and for $j\geq 0$, $v_j \equiv z_j \mod c$ and $0<z_j\leq (2L)^2<c$. Therefore there is $a_i>0$ such that
$$
z_i\equiv v_i = v_{i-1}a_i + v_{i-2} \mod c,
$$
i.e. \refer{seqzi} holds for $i$.

\medskip
It remains to show that with $a,b,c$ defined in this way, we can find an $\linlang[3]$-formula which defines $x\mapsto x^2$ on $[0,2L]$ in $\stru{\M^+}{a,b,c}$.

Let $\gamma(u,v)$ be the $\linlang[3]$-formula
$$
\gamma(u,v)\!:\ \ (\forall u',0<u'\leq u)(\forall v',0\leq v' \leq a)((u,v)\neq(u',v')\limp|au-bv|<|au'-bv'|).
$$
(This is $0^*$ from Corollary \ref{corapprox} without the bound $u'\leq
b$.) 
Then the $\linlang[3]$-formulas 
\begin{eqnarray}
V(v)\!:& & (\exists u, 0<u\leq b)\gamma(u,v), \\
V_0(v)\!:& & (\exists u, 0<u\leq b)(\gamma(u,v) \land  au-bv>0),\\
V_1(v)\!:& & (\exists u, 0<u\leq b)(\gamma(u,v) \land  au-bv\leq 0),
\end{eqnarray}
define the sets $V=\set{v_i}{0\leq i<4L}$, $V_0=\set{v_i}{0\leq i<4L$ and $ i $ even$}$ and $V_1=\set{v_i}{0\leq i<4L$ and $ i $ odd$}$ respectively in $\stru{\M^+}{a,b,c}$. For $V$, this follows directly from Corollary \ref{corapprox}. The cases of $V_0$ and $V_1$ are implied by \refer{eq:riuivi} with the case $au-bv=0$ falling into $V_1$, as this is only possible for $i=4L-1$ which is odd.

Notice also that since the sequence $(v_i)$ is increasing, we can define the set of all pairs $(v_{2i},v_{2i+1})$ with $i<2L$ by:
$$
\pi(v,v')\!:\ \ V_0(v) \land V_1(v') \land \lnot(\exists w, v<w<v')V(w). 
$$

Finally, we define $x\mapsto x^2$ on $[0,2L]$ by:
$$
x^2=y \lekv x=y=0 \lor (\exists v,v') (0\leq v,v' \leq a \land\pi(v,v') \land x= v\!\!\!\!\mod c \land y = v'\!\!\!\!\mod c),
$$
where, of course, $z= w\!\!\mod c$ stands for $0\leq z < c \land
(\exists m)( 0 \leq m\leq w\land w=z+cm)$.

\medskip

We denote the formula on the right hand side of  the previous
equivalence by $\sigma(x,y)$. Notice that this is a bounded formula
(even bounded by constant terms) and does not contain parameters from
$M$. The same holds true about the formula which defines multiplication on $[0,L]^2$ using \refer{eq:multfromsquare}, as it is equivalent to
$$
\mu(x,y,z)\!:\ \ (\exists p,q,r)( 0\leq p,q,r < c\land\sigma(x,p) \land \sigma(y,q) \land \sigma(x+y,r) \land 2z+p+q=r),
$$
where all three numbers $r=(x+y)^2$, $p=x^2$, $q=y^2$ can be bounded by $(2L)^2$ and therefore by $c$.

\begin{rmk}
Note that in the above construction we did not use any specific property of the constructed squaring function besides that it is nonzero, its range is bounded in $\M$ and that it is coded in $\M$ (via G\"odels coding of finite sets). A slight modification of this construction could be therefore used to yield a nonstandard segment of any unary function $f$ (or, with only little more modifications, even any $n$-ary relation $R$ for arbitrary $n$) on $M$ that is coded in $\M$. (The requirement of being nonzero can be easily overcome by consructing $f+1$ instead of $f$ and subtracting $1$ at the end, and of course any coded relation is bounded in all coordinates in $M$.) 
\end{rmk}

\subsection{A wild model of $\LA[2]$}
\label{sect:modLA2}

Let $\stru{\M^+}{a,b,c}$ be the model from the previous subsection. We
will show that the multiplication on $[0,L]^2$ is $\emptyset$-definable in
the structure $\stru{\M^+}{ac,abc^2+c}\models\LA[2]$ (where again $ac$
and $abc^2+c$ stand for the functions of scalar multiplication by
these two elements). If we could prove that scalar multiplication by
$a$, $b$ and $c$ is definable using scalar multiplication by $ac$ and
$abc^2+c$, we would be done, but this is not the case. We are only
able to define the elements $a1$, $b1$ and $c1$. We would also be
done, if we could define scalar multiplication by $ac$, $bc$ and $c$,
because the formula defining partial multiplication is homogeneous in
$a$ and $b$. We do have $ac$, but we are not able to define $bc$ and
$c$. However, what we can do is to define scalar multiplication by
$bc$ and $c$ \emph{restricted to the interval} $[0,a1]$, which 
suffices for our purpose. 

Let $\alpha^*x:=acx$ for all $x$, $\beta^*x=bcx$ for $0\leq x \leq a$,
$\gamma^*x= cx$ for $0\leq x \leq a$ and let $\gamma^*x=\beta^*x=0$
for $x>a$. We will modify the formula $\mu$ defined above as
follows. We keep $a1$, $b1$ and $c1$ in the inequalities that
determine the range of quantification. (In the formula we used just
letters $a$, $b$ and $c$ for the sake of brevity; now we have to be
more careful.) We replace the other occurrences of scalar
multiplication by $a$, $b$ and $c$ by the functions $\alpha^*$,
$\beta^*$ and $\gamma^*$ respectively. We will denote the resulting
formula by $\mu'(x,y,z)$.

%
%
%


First we prove that $\mu(x,y,z)\IFF\mu'(x,y,z)$ for all $x,y,z\in M$. In what follows, for a
subformula $\varphi$ of $\mu$, we denote by $\varphi'$ the
corresponding subformula of $\mu'$.




We observe that during the evaluation of $\mu(x,y,z)$, the subformulas
$V(v)$, $V_0(v)$ and $V_1(v)$ are only evaluated for $0\leq v \leq
a$. For $0\leq v,v' \leq a$ (and any $u,u'$), we have 
\[
|au-bv|<|au'-bv'| \IFF |\alpha^*u-\beta^*v|<|\alpha^*u'-\beta^*v'|
\]
and similarly for $au-bv>0$ and $au-bv\leq 0$. Therefore, for $0\leq v
\leq a$, we get 
\[
V(v) \IFF V'(v)
\]
and the same for $V'_0$, $V'_1$. Consequently, for $0\leq u,v \leq
a$, 
\[
\pi(u,v) \IFF \pi'(u,v). 
\]
Further, for $0\leq w \leq a$ and any $z$, we get that 
$$z=w\!\!\mod c \IFF (z=w\!\!\mod c)'.$$
(Note that $z=w\!\!\mod c$ means $z\equiv w \mod c \land 0\leq z < c1$.)
From this, it is easy to see
that the same equivalence holds true also for $\sigma(x,y)$ and consequently for $\mu(x,y,z)$.

\medskip

It remains to find definitions of $\beta^*,\gamma^*,a1,b1,c1$ in $\stru{\M^+}{ac,abc^2+c}$. Let us denote by $\alpha = ac$, $\beta = bc$. Then $abc^2+c = \alpha\beta+c$. 

%
To define $\gamma^*$, we first define an auxiliary function $\gamma^\circ$ by
\[
\gamma^\circ x = ((\alpha\beta+c)x)\!\!\!\!\mod\alpha.
\]
Notice that for $0\leq x < a1$, $\gamma^\circ x = cx = \gamma^*x$, but $\gamma^\circ a = 0 \neq \gamma^*a$. This enables us to write down parameter-free definitions of $a1,c1$ and $\gamma^*$ in $\stru{\M^+}{\alpha,\alpha\beta+c}$:

\begin{eqnarray*}
a1 &=& \min\set{x>0}{\gamma^\circ x = 0},\\
c1 &=& \gamma^\circ 1,
\end{eqnarray*}

and

$$
\gamma^*x = 
\begin{piecewise} 
\gamma^\circ x & \textrm{for\ } 0\leq x < a1,\\
\alpha 1& \textrm{for\ } x = a1,\\
0 & \textrm{otherwise.} 
\end{piecewise}
$$

To define $\beta^*$, we again start with a definition of an auxiliary function $\beta^\circ$ 
$$
\beta^\circ x = ((\alpha\beta+c)x)\ \mathrm{div}\ \alpha,
$$
where the  function $ u\ \mathrm{div}\ \alpha$ is defined by $w= u\ \mathrm{div}\ \alpha \lekv \alpha w \leq u < \alpha(w+1)$.
Again, it is not difficult to see that $\beta^\circ x = bcx = \beta^*x $ for $0\leq x < a$, but $\beta^\circ a = abc+1 \neq \beta^*a$, which we can use to $\emptyset$-define $b1$ and $\beta^*$ in $\stru{\M^+}{\alpha,\alpha\beta+c}$ as follows:

$$
b1 = (\beta^\circ a)\ \mathrm{div}\ \alpha,
$$

and 

$$
\beta^*x = 
\begin{piecewise} 
\beta^\circ x & \textrm{for\ } 0\leq x < a1,\\
\beta^\circ x - 1& \textrm{for\ } x = a1,\\
0 & \textrm{otherwise.} 
\end{piecewise}
$$

Finally note that the $\linlang[2]$-definition $\mu''(x,y,z)$ of
multiplication on $[0,L]^2$ in
$\stru{\M^+}{\alpha,\alpha\beta+c}$ that we just constructed 
is a bounded formula, because it was constructed from $\mu'$ by
substituting definitions of functions $\alpha^*,\beta^*,\gamma^*$ and
constants $a1,b1,c1$ to appropriate places. It can easily be seen
that during the evaluation of $\mu'$ in
$\stru{\M^+}{\alpha^*,\beta^*,\gamma^*}$ the function $\alpha^*x$ is
evaluated only for $0\leq x\leq b$ and $\beta^*x,\gamma^*x$ only for
$0\leq x \leq a$. Therefore, always
$0\leq\alpha^*x,\beta^*x,\gamma^*x<abc1<(\alpha\beta+c)1$. Clearly
also $0\leq a1,b1,c1<(\alpha\beta+c)1$ and thus the existentially
quantified variables in  $\mu''$ can be bounded by $(\alpha\beta+c)1$.

\section{A non-NIP discretely ordered module}
\label{sect:nonNIP}

A structure $\A$ is NIP (not independence property; see \cite{NIPguide} for an extensive introduction to NIP structures and theories) if there is no formula $\varphi(\vect{x},\vect{y})$ such that for every $n\in\omega$, there are $\vect{a_i}\in A^{l(\vect{x})}$, with $i<n$, and $\vect{b_J}\in A^{l(\vect{y})}$, with $J\subseteq n$, such that 
$$
\varphi(\vect{a_i},\vect{b_J}) \IFF i\in J.
$$

\medskip

Chernikov and Hils \cite[Question 5.9.1]{ChH14} asked whether all ordered modules are NIP. We answer their question negatively:

Let $\stru{\M^+}{a,b}\models\LA[2]$ be a model in which a multiplication (function $\cdot$ satisfying $x\cdot 0 = 0$ and $x(y+1)=xy+x$) is definable on $[0,L]^2$ for some non-standard $L$ (such models exist by Theorem \ref{thm:main}) and let $\A=\stru{M}{0,+,-,\leq,r}_{r\in R}$ be the discretely ordered module corresponding to $\stru{\M^+}{a,b}$ (see subsection \ref{subsect:modules}).

\begin{prop}
The discretely ordered module $\A$ is not NIP.
\end{prop}

\begin{proof}

Let $\psi$ define multiplication $\cdot$ on $[0,L]^2$ in
$\stru{\M^+}{a,b}$. We construct a formula $\psi'$ by replacing
possible occurrences of the constant $1$ in $\psi$ by the definition of
$1$ in $\A$ (the least positive element), replacing every quantifier
$(Qx)$ by $(Qx, 0\leq x)$ and replacing scalars $a,b$ by scalars $q,r\in R$
representing the same functions on $M$ as $a,b$ do. Then the formula $x,y,z\geq 0 \land \psi'(x,y,z)$ defines $\cdot$ on $[0,L]^2$ in $\A$. 

Clearly $\cdot\restr \N^2$ is the usual multiplication on $\N$ and the
formula $\varphi(x,y)\!:\ (\exists z, 0\leq z \leq y)z\cdot x = y$ when restricted to $\N^2$ defines the usual divisibility relation.

It is now easy to prove that the $\varphi$ has the independence property:
Let $n\in\omega$ be given. For $i<n$ take $a_i$ the $i$-th prime number in $\N$ and for $J\subseteq n$ take $b_J=\prod_{i\in J}a_i$.
\end{proof}

\section{Open problems}

What is the strongest possible quantifier elimination result for $\LA[\kappa]$ with $\kappa\geq 2$? Can definable sets in models of $\LA[\kappa]$ be precisely characterized?

Is there a model of $\sLA[\kappa]$ with $\kappa\geq 2$ whose first-order theory is model complete/NIP? Can those models be characterized?

\bibliography{biblio_final}{}
\bibliographystyle{amsalpha}

\end{document}